\documentclass[12pt]{amsart}
\usepackage{epsfig,color,wrapfig}

\makeatother

\theoremstyle{definition}

\newtheorem{thm}{Theorem}[section]
\newtheorem{cor}[thm]{Corollary}

\newtheorem{rmk}[thm]{Remark}
\newtheorem{prop}[thm]{Proposition}

\numberwithin{equation}{section}

\address{National Taiwan University, Department of Mathematics\\
No.1, Sec. 4, Roosevelt Rd., Da'an Dist., Taipei City, Taiwan, 10617}
\email{jechang@ntu.edu.tw}

\title{Self similar solutions to the inverse mean curvature flow in $\mathbb{R}^2$}
\author{Jui-En Chang} 
\date{\today}

\usepackage{latexsym, amsmath,amssymb}
\usepackage{amsmath}
\usepackage{amsfonts}
\usepackage{amssymb}
\usepackage{amscd}
\usepackage{amsthm}
\usepackage{amsbsy}
\usepackage{graphicx}
\usepackage{bm}
\usepackage{epsfig,epstopdf,url,array}

\begin{document}
\maketitle
\begin{abstract} 
In this paper, we obtain a complete list of all self-similar solutions of inverse mean curvature flow in $\mathbb{R}^2$.
\end{abstract}

\section{Introduction}
The inverse mean curvature flow is given by
\begin{equation}
\frac{dx}{dt}=\frac{1}{H}N,
\end{equation}
where $N$ is a unit normal vector to the surface and the mean curvature $H$ is defined as the trace of $\nabla N$. This flow is related to several different important topics. Huisken and Ilmanen \cite{HI2}, \cite{HI3} use the inverse mean curvature flow to establish the Riemannian Penrose inequality. Bray and Neves \cite{BN} also use inverse mean curvature flow to proof the Poincare conjecture for certain 3-manifolds.

The solutions which move by expanding outward, which are called self-expanders, are studied extensively. Drugan, Lee, and Wheeler \cite{DLW} give an explicit classification of the one-dimensional homothetic solitons and establish that families of cycloids are the only translating solitons. For the hypersurfaces in $\mathbb{R}^n$, Huisken and Ilmanen in \cite{HI} used a phase-plane analysis to establish homothetic solitons which are topological hyperplanes. In \cite{DLW}, Drugan, Lee, and Wheeler also construct rotationally symmetric hypercylinder expanders and the infinite bottles, which is asymptotic to coaxial cylinders with different radius at each end. The rigidity of the infinite bottle is established in by Drugan, Fong and Lee in \cite{DFL}. For the higher codimension homothetic solutions, Castro and Lerma \cite{CL} establish that the compact soliton should be minimally embedded in a sphere. They also study the rotationally invariant Lagrangian homothetic solitons.

In $\mathbb{R}^2$, the self-similar motions are translating, shrinking, expanding, rotating and a combination of the above. The previous work by Drugan, Lee and Wheeler \cite{DLW} covers translation and scaling(shrinking and expanding). Up to now, the solutions concerning general self-similar motions in $\mathbb{R}^2$ are not studied. Here, we study the solutions which flow by any self-similar motions. In this paper, up to rotation and scaling, we find explicitly the self similar solutions corresponds to each self-similar motion.

In order to make the expression of rotation simpler, from now on, we identify $\mathbb{R}^2$ with $\mathbb{C}$. Therefore, the rotation can be understand as multiplying a complex number with unit absolute value and the scaling can be understand as multiplying a positive real number.

\begin{thm}
If a solution of the inverse mean curvature flow moves by scaling and rotation, there exists $c,d$ such that it can be expressed as $e^{(d+ci)t}x(s)$ for some initial curve $x(s)$ satisfying
\begin{equation}
c\langle x,T\rangle-d\langle x,N\rangle=\frac{1}{k}.
\end{equation}
\end{thm}
We can think of the constants $c$ and $d$ in the theorem as the rotation speed and the expanding speed of the solution.

\begin{prop}
If a solution of the inverse mean curvature flow moves by scaling and rotation with speed $d$ and $c$ respectively, any rotation and scaling is also a self similar solution with the same scaling and rotation speed.
\end{prop}

For the solution move by translation, rotation and scaling, we can use the following theorem to reduce it to the case with no translating movements.
\begin{thm}
If the solution moves by rotation, scaling and translation with either the rotation part or the scaling part nonzero, we can choose a new origin so that the solution moves only by rotation and scaling in the new coordinates.
\end{thm}

Therefore, we can reduce to the case concerning only rotation and scaling. The following is the main theorem concerning the behavior of the self similar solutions.

\begin{thm}
For the self-similar solutions with the rotation speed and expanding speed $c,d$, if $(c,d)\neq(0,0)$, we can separate the behavior of the solution into 3 cases:
\begin{itemize}
\item Case 1: Undercritical case($c^2<4(1-d)$). The solution is unique up to rotation and scaling. It behaves like a cycloid which moves along a logarithmic spiral or a circle. There will be infinitely cusps on the solution.
\item Case 2: Critical case($c^2=4(1-d)$). Up to rotation and scaling, there are 2 different solutions. One of them is a logarithmic spiral, the other one has a cusp and is asymptotic to the logarithmic spiral near the origin.
\item Case 3: Overcritical case($c^2>4(1-d)$). Up to rotation and scaling, there are 4 different solutions. 2 of them are logarithmic spirals, the other 2 are asymptotic to the logarithmic spirals in both end of the solution. One of the latter 2 solutions has a cusp, the other is smooth.
\end{itemize}
Explicit parametrizations will be summarized at the end of this paper.
\end{thm}

Most solution are either incomplete or has singularities. However, in some cases, there exists smooth complete solutions.

\begin{cor}
When $d\geq1$, $(c,d)\neq(0,1)$, the smooth solution other than the logarithmic spirals is the unique complete, non-compact solution without singularity in this case.
\end{cor}

\begin{rmk}
From the result of \cite{CL}, \cite{DFL}, \cite{DLW} and \cite{HI} for an $n$-dimensional homothetic soliton, the behavior is highly related to the expansion speed $d$. For $d=\frac{1}{n}$ we have compact solutions and for $d>\frac{1}{n}$ there exist complete, non-compact solutions. Here, we see that this also holds even if we consider of the effect of rotation.
\end{rmk}

The paper will be organized as follows: In section 2, we proof the theorem concerning the behavior the self-similar solutions in $\mathbb{R}^2$ and the equation concerning the curvature. In section 3, we set up an ODE describing the curves. We can get explicit equations for all the solutions. The behavior of the curves is determined by the sign of $c^2-4(1-d)$. In section 4, 5, 6, we study the critical, overcritical, undercritical cases, respectively. Section 7 is a summary of the behavior of the solutions.

Throughout this paper, we use the following convention: First, we neglect all the constants arise from indefinite integration. All of such constants correspond to rotation and scaling. Therefore, the final curve we derived in each case is unique up to a rotation and a scaling. Second, up to a rotation, $\theta$ will always stand for the direction of the unit tangent vector of the self-similar curve.

\section{Self-similar motions in $\mathbb{R}^2$}

Let $x:\mathbb{R}\to\mathbb{R}^2\cong\mathbb{C}$ be a curve parametrized by the arc length $s$. Any self-similar motion is a combination of rotation, scaling and translation. We can express a continuous self-similar motion as
\begin{equation}
x(s,t)=\mathbf{S}(t)e^{i\mathbf{R}(t)}x(s)+\mathbf{T}(t),
\end{equation}
where $\mathbf{S}:\mathbb{R}\to\mathbb{R}^+$, $\mathbf{R}:\mathbb{R}\to\mathbb{R}$, and $\mathbf{T}:\mathbb{R}\to\mathbb{R}^2\cong\mathbb{C}$ stand for the scaling, rotation and translation, respectively. For the initial condition, we have $\mathbf{S}(0)=1$, $\mathbf{R}(0)=0$, $\mathbf{T}(0)=0$. Let us consider the case without translation first.

\begin{proof}[Proof of theorem 1.1]
From the defining equation of inverse mean curvature flow, we have
\begin{equation}
\langle\partial_tx(s,t),N(s,t)\rangle=-\frac{1}{k}.
\end{equation}
Here, $N(s,t)$ is the unit normal vector obtained by rotating the tangent vector $T$ counter-clockwise by $\frac{\pi}{2}$. We have $N(s,t)=e^{i\mathbf{R}(t)}N(s)$. The signed curvature with respect to $N(s,t)$ is defined as $k(s,t)=\mathbf{S}(t)^{-1}\langle \frac{dT}{ds},N \rangle=-H$. Note that $k(s,t)=k(s)\mathbf{S}(t)^{-1}$ is the curvature after scaling.
\begin{equation}
\begin{split}
-\frac{\mathbf{S}(t)}{k(s)}&=-k(s,t)^{-1}=\langle \partial_tx(s,t), N(s,t)\rangle\\
&=\langle \mathbf{S}'(t)e^{i\mathbf{R}(t)}x(s)+i\mathbf{R}'(t)\mathbf{S}(t)e^{i\mathbf{R}(t)}x(s),N(s,t)\rangle\\
&=\mathbf{S}'(t)\langle x(s), N(s)\rangle-\mathbf{R}'(t)\mathbf{S}(t)\langle x(s),T(s)\rangle.
\end{split}
\end{equation}
Divide both side by $\mathbf{S}(t)$ yields
\begin{equation}
-\frac{1}{k(s)}=\frac{\mathbf{S}'(t)}{\mathbf{S}(t)}\langle x(s), N(s)\rangle-\mathbf{R}'(t)\langle x(s),T(s)\rangle.
\end{equation}
Note that the left hand side does not depend on $t$. Set $c=\mathbf{R}'(0)$ and $d=\frac{\mathbf{S}'(0)}{\mathbf{S}}$.
\begin{equation}
d\langle x(s), N(s)\rangle-c\langle x(s),T(s)\rangle=-\frac{1}{k(s)}=\frac{\mathbf{S}'(t)}{\mathbf{S}(t)}\langle x(s), N(s)\rangle-\mathbf{R}'(t)\langle x(s),T(s)\rangle.
\end{equation}

Solve $\mathbf{R}'(t)=c$, $\frac{\mathbf{S}'(t)}{\mathbf{S}(t)}=d$. We have $\mathbf{R}(t)=ct$, $\mathbf{S}(t)=e^{dt}$. The scaling and the rotation can be written as $\mathbf{S}(t)e^{i\mathbf{R}(t)}=e^{(d+ci)t}$.
\end{proof}
\begin{rmk}
If the functions $\langle x(s), N(s)\rangle$ and $\langle x(s), T(s)\rangle$ are dependent, it is possible that $c$, $d$ are not be unique. For example, the solution $x(s)=(e^s\cos s, e^s\sin s)$ can be think of as a solution which moves by rotation with $(c,d)=(2,0)$ or a solution which moves by scaling with $(c,d)=(0,5)$.
\end{rmk}

From the equation characterize the self similar solution, we immediate have the following.
\begin{proof}[Proof of proposition 1.2]
Let $x(s)$ be a curve which satisfies $d\langle x(s), N(s)\rangle-c\langle x(s),T(s)\rangle=-\frac{1}{k(s)}$. Fix $\mathbf{S}\in\mathbb{R}^+$, $\mathbf{R}\in\mathbb{R}$. Consider the curve $\bar{x}(s)=\mathbf{S}e^{i\mathbf{R}}x(s)$. The quantities $\langle x(s), N(s)\rangle$ and $\langle x(s), T(s)\rangle$ is unchanged after rotation and becomes $\mathbf{S}$ times larger after scaling. On the other hand, the curvature of $\bar{x}$ is $\frac{k(s)}{\mathbf{S}}$. Therefore, the curve satisfies $d\langle \bar{x}(s), N_{\bar{x}}(s)\rangle-c\langle \bar{x}(s),T(s)_{\bar{x}}\rangle=-\frac{1}{k_{\bar{x}}(s)}$ and it is a solution with the same rotation and expanding speed.
\end{proof}
Now, let us consider the self-similar motion with a translating term.
\begin{proof}[Proof of theorem 1.3]
From the same calculation as in the previous proof, we have
\begin{equation}
\begin{split}
-\frac{1}{k(s)}&=\frac{\mathbf{S}'(t)}{\mathbf{S}(t)}\langle x(s), N(s)\rangle-\mathbf{R}'(t)\langle x(s),T(s)\rangle+\langle\frac{\mathbf{T}'(t)}{\mathbf{S(t)}},N(s,t)\rangle.
\end{split}
\end{equation}
The left hand side is independent of $t$. Again, we can set $c=\mathbf{R}'(0)$ and $d=\frac{\mathbf{S}'(0)}{\mathbf{S}}$. From the assumption, $(c,d)\neq(0,0)$. Set $t=0$ and put $\bar{x}(s)=x(s)+\frac{\mathbf{T}'(0)}{d-ci}$. We have
\begin{equation}
-d\langle \bar{x}(s), N(s)\rangle+c\langle \bar{x}(s),T(s)\rangle=\frac{1}{k(s)}.
\end{equation}
Therefore, the solution $x$ move by combination of rotation and scaling with respect to the point $-\frac{\mathbf{T}'(0)}{d-ci}$ with rotation speed $c$ and expanding speed $d$.
\end{proof}

\begin{rmk}If the solution move by translation only, without loss of generality, we can rotate the solution such that the solution translates upward. We have $\mathbf{R}(t)=0$, $\mathbf{S}(t)=1$, $\mathbf{T}(t)=it$. In the work of Drugan, Lee and Wheeler \cite{DLW}, this solution is shown to be the cycloid
\begin{equation}
x(s')=\frac{1}{4}(s'-\sin(s'),1-\cos(s')).
\end{equation}
\end{rmk}

\section{Setting up and solving the ODE system for solutions}
From now on, we focus on the equation
\begin{equation}
-d\langle x(s), N(s)\rangle+c\langle x(s),T(s)\rangle=\frac{1}{k(s)}.
\end{equation}
Put $\tau=\langle x,T\rangle$, $\nu=\langle x,N\rangle$. Both $\tau$ and $\nu$ are functions of $s$. We have $c\tau-d\nu=\frac{1}{k}$ and obtain  an ODE system about $\tau$ and $\nu$.
\begin{equation}
\Bigg \{
\begin{array}{l}
\frac{d\tau}{ds}=1+k\nu=1+\frac{\nu}{c\tau-d\nu},\\
\frac{d\nu}{ds}=-k\tau=-\frac{\tau}{c\tau-d\nu}. \end{array}
\end{equation}
Note that whenever $c\tau-d\nu=0$, the curvature $k$ is infinity and we have a singularity of the curve. We are interested in the smooth part. Set a new variable
\begin{equation}
\bar{s}=\int\frac{1}{c\tau-d\nu}ds.
\end{equation}
By changing the variable, we have
\begin{equation}
\Bigg \{
\begin{array}{l}
\frac{d\tau}{d\bar{s}}=c\tau-d\nu+\nu=c\tau+(1-d)\nu,\\
\frac{d\nu}{d\bar{s}}=\tau. \end{array}
\end{equation}
This dynamic system has the same flow line as the original one. This is a linear ODE system, we can use this to study the flow line and then analyse the behavior of the singularity separately.

Since the system is linear, we can use polar coordinates for $\tau$ and $\nu$ to get simpler expressions. Set $\nu=r\cos\phi$, $\tau=r\sin\phi$. Note that we have $\frac{dr}{ds}=\frac{\tau}{r}=\sin\phi$. We want to find a relation between $r$ and $\phi$ on a flow line.

\begin{prop}
When $\phi$ is a constant $\phi_0$, it satisfies $\tan^2\phi+c\tan\phi+(1-d)=0$. In this case, up to a rotation, the solution is the logarithmic spiral $x=e^{(-\tan\phi_0+i)\theta}$.
\end{prop}
\begin{proof}
For the solution with constant $\phi$, the flow vector and the position vector should be parallel. Therefore,
\begin{equation}
(r\cos\phi,r\sin\phi)\parallel \big( -\sin\phi,c\sin\phi+(1-d)\cos\phi \big).
\end{equation}
Which is equivalent to 
\begin{equation}
-\sin^2\phi=c\sin\phi\cos\phi+(1-d)\cos^2\phi.
\end{equation}
Suppose $\cos\phi=0$, the right side vanishes while the left side equals -1. It is a contradiction. Therefore $\cos\phi\neq0$, we can divide by $\cos^2\phi$ and get the equation
\begin{equation}
\tan^2\phi+c\tan\phi+(1-d)=0.
\end{equation}
Let $\phi_0$ be the solution of the equation. We have $(\tau,\nu)=(r\sin\phi_0, r\cos\phi_0)$. Now, we want to solve for the original curve $x$. Let $\theta$ be the angle of the tangent vector $T$. We have $T=e^{i\theta}$, $N=i e^{i\theta}$. Use $k=\frac{1}{c\tau-d\nu}=-\frac{\cos\phi_0}{r}$,
\begin{equation}
\theta=\int kds=\int-\frac{1}{r}\cot\phi_0dr=-\cot\phi_0\log r.
\end{equation}
From $x=\tau T+\nu N$, we can construct the solution.
\begin{equation}
\begin{split}
x&=r(\sin\phi_0 +i\cos\phi_0)e^{i\theta}
=(\sin\phi_0 +i\cos\phi_0)e^{(-\tan\phi_0+i)\theta}\\
\end{split}
\end{equation}
which is a logarithmic spiral. Up to a rotation, we can write it as $x=e^{(-\tan\phi_0+i)\theta}$.
\end{proof}

The line $\phi=\phi_0$ is a flow line. From the uniqueness of flow line, any other flow line will never touches the line and we have $\frac{d\phi}{d\bar{s}}\neq0$. In this case, $r$ is always a function of $\phi$. For the point with $\frac{d\nu}{d\bar{s}}\neq0$, we have
\begin{equation}
\frac{d\tau}{d\nu}=
\frac{r'\sin\phi+r\cos\phi}{r'\cos\phi-r\sin\phi}=-c-(1-d)\cot\phi,
\end{equation}
where $r'$ means derivative with respect to $\phi$. Therefore,
\begin{equation}
\frac{r'}{r}=\frac{-d\cos\phi\sin\phi+c\sin^2\phi}{1+c\cos\phi\sin\phi-d\cos^2\phi}=\frac{c\tan^2\phi-d\tan\phi}{\tan^2\phi+c\tan\phi+(1-d)}.
\end{equation}
If $\frac{d\nu}{d\bar{s}}=\tau=0$, this can only happen on the $\nu$-axis. If $d=0$, the $\nu$-axis is a flow line with constant $\phi$. Otherwise, the flow vector is perpendicular to the position. In this case, we also have $r'=0=\frac{c\tan^2\phi-d\tan\phi}{\tan^2\phi+c\tan\phi+(1-d)}$.

Now, We can integrate both side by substituting $t=\tan\phi$,
\begin{equation}
\begin{split}
\log r&=\int\frac{c\tan^2\phi-d\tan\phi}{\tan^2\phi+c\tan\phi+(1-d)}d\phi=\int\frac{ct^2-dt}{\big(t^2+ct+(1-d)\big)(t^2+1)}dt\\
&=\frac{1}{2}\log\Big|\frac{t^2+1}{t^2+ct+(1-d)}\Big|+\frac{c}{2}\int\frac{1}{t^2+ct+(1-d)}dt.
\end{split}
\end{equation}
The second term has a geometric interpretation. Using $k=\frac{1}{c\tau-d\nu}$, $\frac{d\phi}{ds}=\frac{1}{r}(\cos\phi+\frac{1}{c\sin\phi-d\cos\phi})$, the direction $\theta$ of the tangent vector $T$ is given by
\begin{equation}
\begin{split}
\theta&=\int kds=
\int\frac{\sec^2\phi}{\tan^2\phi+c\tan\phi+(1-d)}d\phi=\int\frac{dt}{t^2+ct+(1-d)}.
\end{split}
\end{equation}
To emphasise the dependence to the constant $c$, $d$. Define
\begin{equation}
\theta_{c,d}(t)=\int\frac{1}{t^2+ct+(1-d)}dt.
\end{equation}
The form of this integration will depend on the sign of $c^2-4(1-d)$. We will postpone this integration to each of the following sections. Here, we derive only the expression of the solution in terms of $r(\phi)$, $\theta_{c,d}(\phi)$ and $\phi$.
\begin{equation}
r=\frac{|\sec\phi|}{\sqrt{|(\tan^2\phi+\frac{c}{2})^2+(1-d-\frac{c^2}{4})|}}\exp\big(\frac{c}{2}\theta_{c,d}(\tan\phi)\big).
\end{equation}

Again, from $x=\tau T+\nu N$, we can construct the solution.
\begin{equation}
\begin{split}
x(\phi)&=r(\sin\phi +i\cos\phi)e^{i\theta_{c,d}(\phi)}\\
&=\frac{1}{\sqrt{|(\tan^2\phi+\frac{c}{2})^2+(1-d-\frac{c^2}{4})|}}e^{(\frac{c}{2}+i)\theta_{c,d}(\phi)}(\tan\phi+i).
\end{split}
\end{equation}
Now, we can compute the explicit solution in each case.

\section{Critical case: $c^2=4(1-d)$}
In this case, the special solution is obtained by solving
\begin{equation}
\tan^2\phi+c\tan\phi+(1-d)=0.
\end{equation}
There is only one solution $\tan\phi_0=-\frac{c}{2}$, up to a rotation, we can write the solution as
\begin{equation}
\begin{split}
x&=e^{(\frac{c}{2}+i)\theta}.\\
\end{split}
\end{equation}

For the other solution,
\begin{equation}
\theta_{c,d}(t)=\int\frac{1}{(t+\frac{c}{2})^2}dt=-\frac{1}{t+\frac{c}{2}}.
\end{equation}
Therefore, $\tan\phi+\frac{c}{2}=-\frac{1}{\theta}$. The curve is given by
\begin{equation}
\begin{split}
x&=\theta e^{(\frac{c}{2}+i)\theta}(\tan\phi+i)
=(-\frac{c}{2}+i)\theta e^{(\frac{c}{2}+i)\theta}-e^{(\frac{c}{2}+i)\theta}.
\end{split}
\end{equation}
The curve can be written as adding a logarithmic spiral $r=e^{\frac{c}{2}\theta}$ with $r=\theta e^{\frac{c}{2}\theta}$ spiral. If we differentiate this curve with respect to $\theta$,
\begin{equation}
\begin{split}
\frac{dx}{d\theta}&=(-\frac{c}{2}+i)e^{(\frac{c}{2}+i)\theta}+(-\frac{c}{2}+i)(\frac{c}{2}+i)\theta e^{(\frac{c}{2}+i)\theta}-(\frac{c}{2}+i)e^{(\frac{c}{2}+i)\theta}\\
&=-ce^{(\frac{c}{2}+i)\theta}-(1+\frac{c^2}{4})\theta e^{(\frac{c}{2}+i)\theta}=-\Big(c+(1+\frac{c^2}{4})\theta\Big)e^{(\frac{c}{2}+i)\theta}.\\
\end{split}
\end{equation}
The curve will be singular when $\theta=-\frac{4c}{4+c^2}$.

\subsection{Example: Numerical solution of the rotation case: $(c,d)=(2,0)$}

\begin{wrapfigure}{r}{0.23\textwidth}
\includegraphics[width=0.23\textwidth]{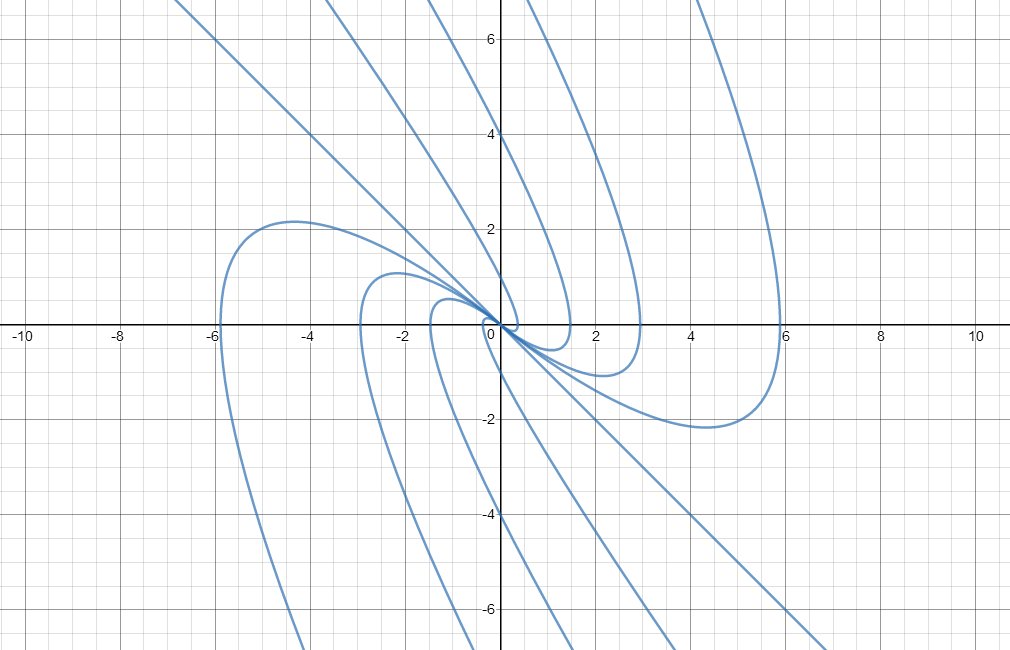}
\end{wrapfigure}
The relation of $\tau-\nu$ is on the right. Up to a rotation, the special solution is given by $\tau=-\nu$. $(\tau,\nu)=(\frac{1}{\sqrt{2}},-\frac{1}{\sqrt{2}})r$ corresponds to the logarithmic spiral $e^{(1+i)\theta}$. Recall that this is also the expanding solution listed in Drugan, Lee and Wheeler's work \cite{DLW}. For the other solution, up to a scaling and a rotation, the curve is as follows.

\begin{figure}[ht]
\centering
\includegraphics[width=150pt]{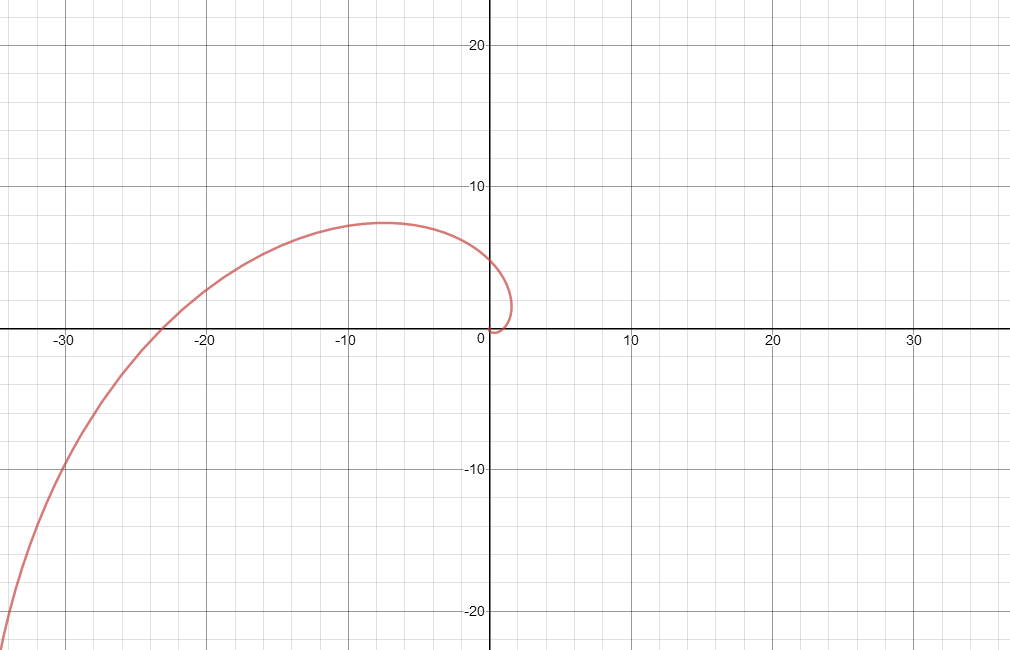}
\includegraphics[width=150pt]{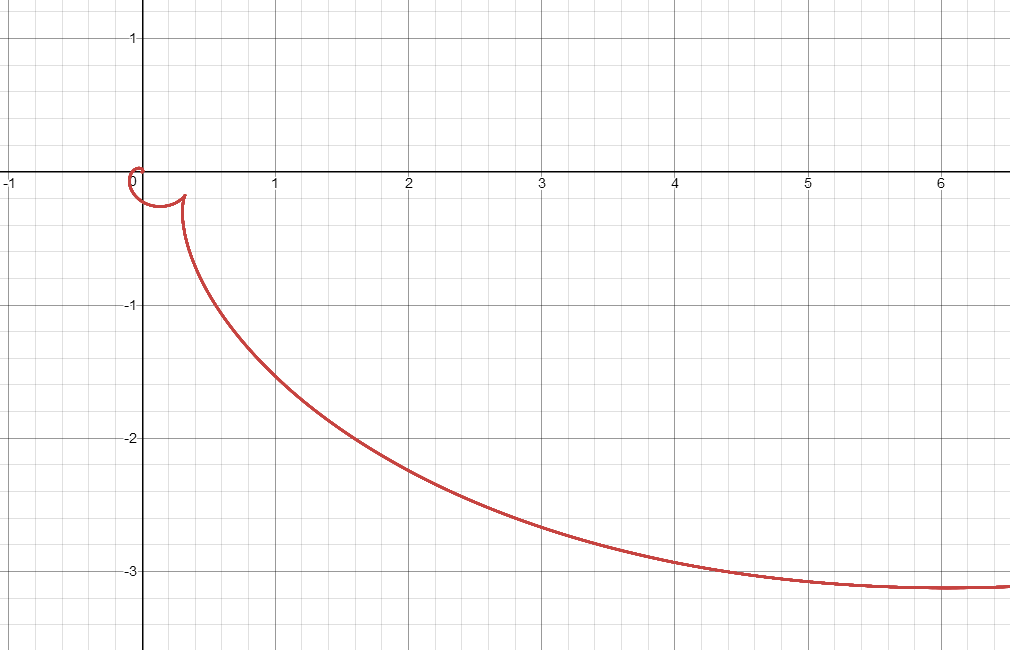}
\caption{The logarithmic spiral and the other solution}
\end{figure}

For the solution other than the logarithmic spiral, there is a singularity of the curve and it separates the curve into 2 parts. One part spiral out to infinity. The other spiral into the origin. Both the inner part and the outer part is asymptotic to the spiral $x=\theta e^{(1+i)\theta}$.

\section{Overcritical case: $c^2>4(1-d)$}

Let $K=\sqrt{\frac{c^2}{4}+d-1}$, $\alpha=\frac{c}{2}+K$, $\beta=\frac{c}{2}-K$.
The solution of
\begin{equation}
\tan^2\phi+c\tan\phi+(1-d)=0
\end{equation}
are given by $\tan\phi=-\alpha,-\beta$. Up to a rotation, this corresponds to the logarithmic spirals 
\begin{equation}
\begin{split}
x&=e^{(\alpha+i)\theta},e^{(\beta+i)\theta}.
\end{split}
\end{equation}

For the general solution, after integration, $\theta_{c,d}$ can be represent by either tanh or coth. In the first case,
\begin{equation}
\theta=\theta_{c,d}(\tan\phi)=\frac{-1}{K}\tanh^{-1}(\frac{1}{K}(\tan\phi+\frac{c}{2})).
\end{equation}
Therefore, $\tan\phi+\frac{c}{2}=-K\tanh(K\theta)$. We have
\begin{equation}
\begin{split}
x=&\frac{1}{\sqrt{|K^2\tanh^2(K\theta)-K^2|}}e^{(\frac{c}{2}+i)\theta}(-(\frac{c}{2}+K\tanh(K\theta))+i)\\
=&\frac{e^{(\frac{c}{2}+i)\theta}}{K}\big(\cosh(K\theta)(i-\frac{c}{2})-K\sinh(K\theta)\big)\\
=&\frac{e^{(\frac{c}{2}+i)\theta}}{2K}\big((e^{K\theta}+e^{-K\theta})(i-\frac{c}{2})-K(e^{K\theta}-e^{-K\theta})\big)\\
=&\frac{1}{2K}\big(e^{(\alpha+i)\theta}(i-\alpha)+e^{(\beta+i)\theta}(i-\beta)\big).
\end{split}
\end{equation}
For simplicity, up to a scaling, the curve can be expressed as
\begin{equation}
\begin{split}
x=&e^{(\alpha+i)\theta}(i-\alpha)+e^{(\beta+i)\theta}(i-\beta).
\end{split}
\end{equation}
This curve can be regarded as adding a logarithmic spiral with another logarithmic spiral.

In the second case,
\begin{equation}
\theta=\theta_{c,d}(\tan\phi)=\frac{-1}{K}\coth^{-1}(\frac{1}{K}(\tan\phi+\frac{c}{2})).
\end{equation}
Therefore, $\tan\phi+\frac{c}{2}=-K\coth(K\theta)$. We have
\begin{equation}
\begin{split}
x=&\frac{1}{\sqrt{|K^2\coth^2(K\theta)-K^2|}}e^{(\frac{c}{2}+i)\theta}(i-(\frac{c}{2}+K\coth(K\theta)))\\
=&\frac{e^{(\frac{c}{2}+i)\theta}}{K}\big(\sinh(K\theta)(i-\frac{c}{2})-K\cosh(K\theta)\big)\\
=&\frac{e^{(\frac{c}{2}+i)\theta}}{2K}\big((e^{K\theta}-e^{-K\theta})(i-\frac{c}{2})-K(e^{K\theta}+e^{-K\theta})\big)\\
=&\frac{1}{2K}\big(e^{(\alpha+i)\theta}(i-\alpha)-e^{(\beta+i)\theta}(i-\beta)\big).
\end{split}
\end{equation}
For simplicity, up to a scaling, the curve can be expressed as
\begin{equation}
\begin{split}
x=&e^{(\alpha+i)\theta}(i-\alpha)-e^{(\beta+i)\theta}(i-\beta).
\end{split}
\end{equation}
This curve can be regarded as adding a logarithmic spiral with another logarithmic spiral with a different angle of rotation.

We can write the solution as
\begin{equation}
\begin{split}
x_\pm=&e^{(\alpha+i)\theta}(i-\alpha)\pm e^{(\beta+i)\theta}(i-\beta).
\end{split}
\end{equation}
If we differentiate with respect to $\theta$,
\begin{equation}
\begin{split}
\frac{d}{d\theta}x_\pm=
&=-\big(e^{\alpha\theta}(1+\alpha^2)\pm e^{\beta\theta}(1+\beta^2)\big)e^{i\theta}.
\end{split}
\end{equation}
Since $e^{\alpha\theta}(1+\alpha^2)+e^{\beta\theta}(1+\beta^2)$ never vanishes, $x_+$ is a smooth curve. On the other hand, when $\theta=\frac{1}{\alpha-\beta}\log(\frac{1+\beta^2}{1+\alpha^2})$, $x_-$ is singular and has a cusp there.

\subsection{Example: Numerical solution of the rotation case: $(c,d)=(3,0)$}

\begin{wrapfigure}{r}{0.23\textwidth}
\centering
\includegraphics[width=0.23\textwidth]{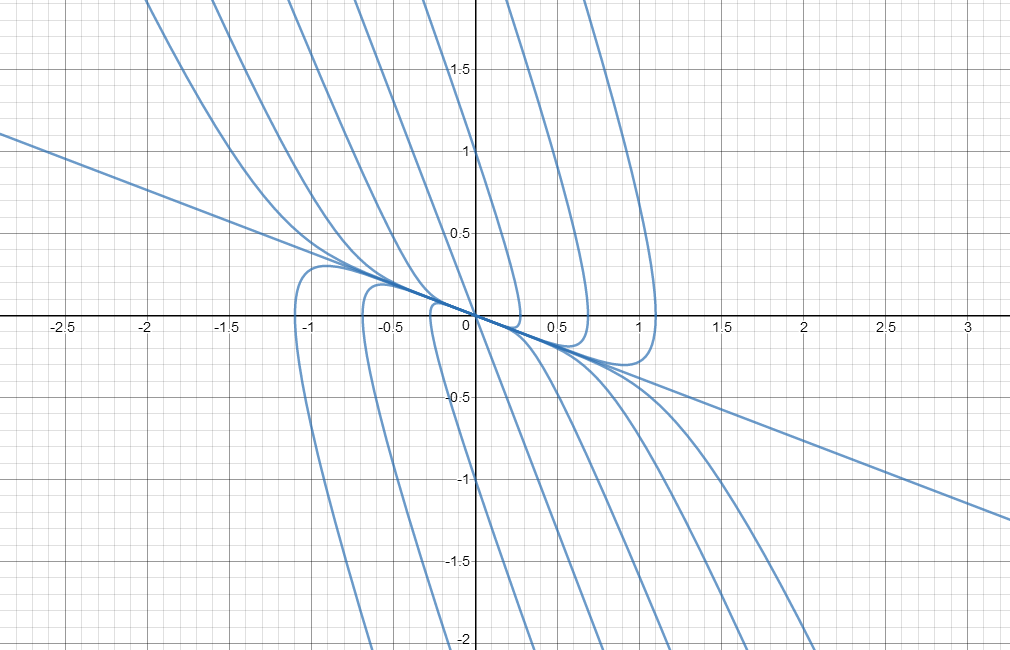}
\end{wrapfigure}

In the $\nu-\tau$ space shown in the right, the 2 straight lines corresponds the 2 logarithmic spiral solutions. These lines cut the plane into different zones which correspond to $-K\tanh(K\theta)=\tan\phi+\frac{c}{2}$ and $-K\coth(K\theta)=\tan\phi+\frac{c}{2}$.

The logarithmic special solutions are given as $x=\exp((\frac{3+\sqrt5}{2}+i)\theta)$, $x=\exp((\frac{3-\sqrt5}{2}+i)\theta)$.

\begin{figure}[ht]
\centering
\includegraphics[width=150pt]{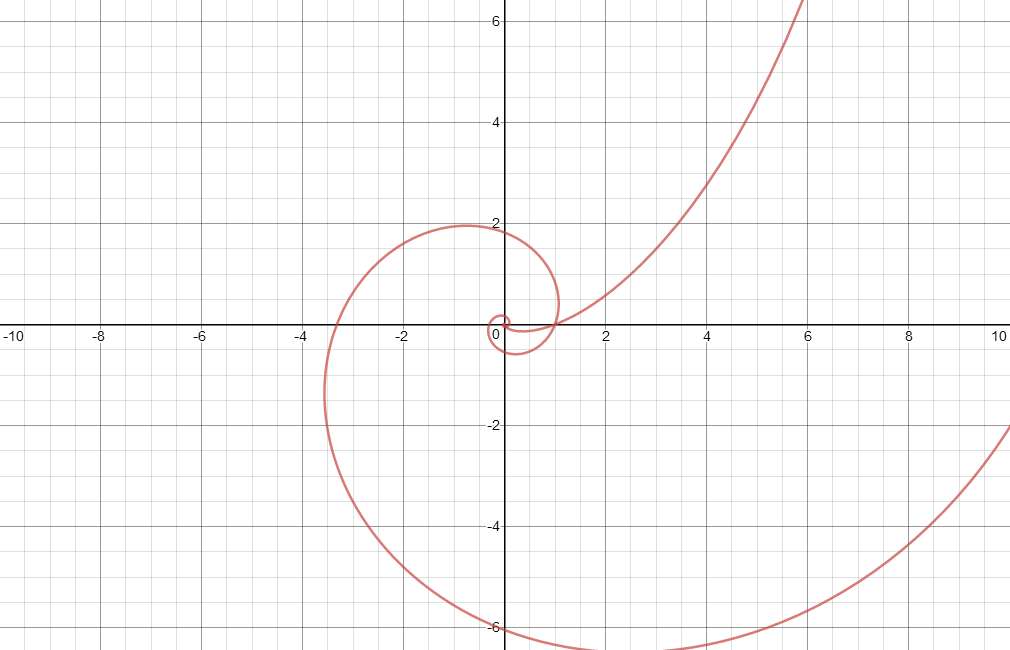}
\includegraphics[width=150pt]{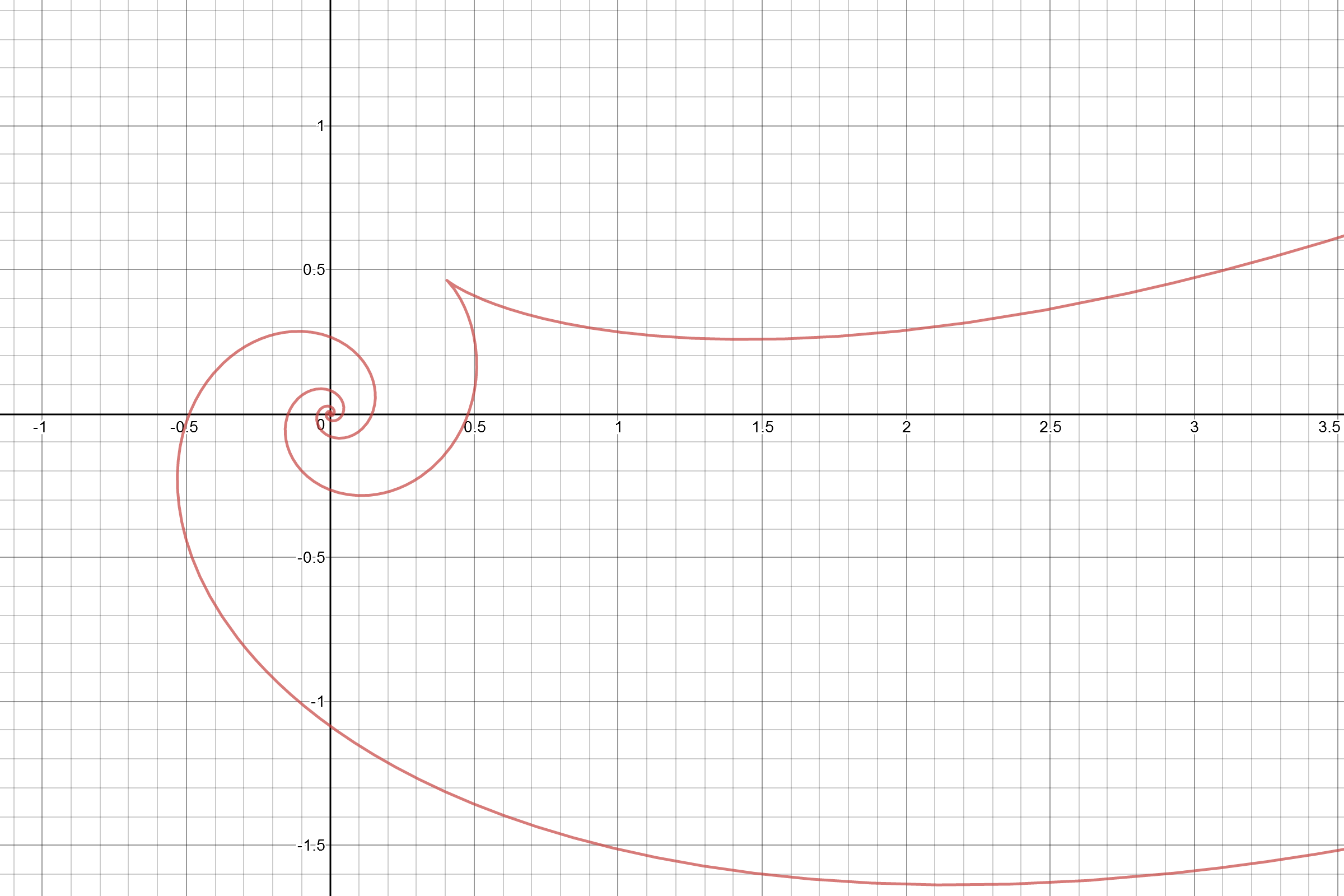}
\caption{The 2 logarithmic spirals and the other 2 solutions}
\end{figure}

For the solutions other than the logarithmic ones, the curve with a cusp corresponds to the $\coth$ function. The curve without a cusp corresponds to the $\tanh$ function. Both curve is asymptotic to the logarithmic spiral $r=\exp(\frac{3-\sqrt{5}}{2}\theta)$ near origin. Both curve is asymptotic to the logarithmic spiral $r=\exp(\frac{3+\sqrt{5}}{2}\theta)$ near infinity.

\subsection{Complete solution when $d\geq1$}

In this case, we have $\alpha\beta=1-d\leq0$. This means $\beta\leq0\leq\alpha$.

$\alpha=\beta=0$ actually belongs to the critical case. In this case, the special solution becomes a circle. Which is compact and embedded.

There are 2 cases left, either one of $\alpha$, $\beta$ vanishes or neither. By symmetry, if one of them vanishes, assume $\beta=0$, the solution becomes
\begin{equation}
x=\frac{1}{c}\big(e^{(c+i)\theta}(i-c)+e^{i\theta}i\big).
\end{equation}
This is a complete solution. On end goes to $r=e^{c\theta}$ logarithmic spiral. The other end converges to a circle. There is no self-intersection of this solution. However, since it approaches a circle when going inward, it is not embedded.

In the $\beta<0<\alpha$ case. Both ends go to infinity and will be asymptotic to logarithmic spirals which rotate in opposite directions. There will be infinitely many self-intersections.

\begin{figure}[ht]
\centering
\includegraphics[width=150pt]{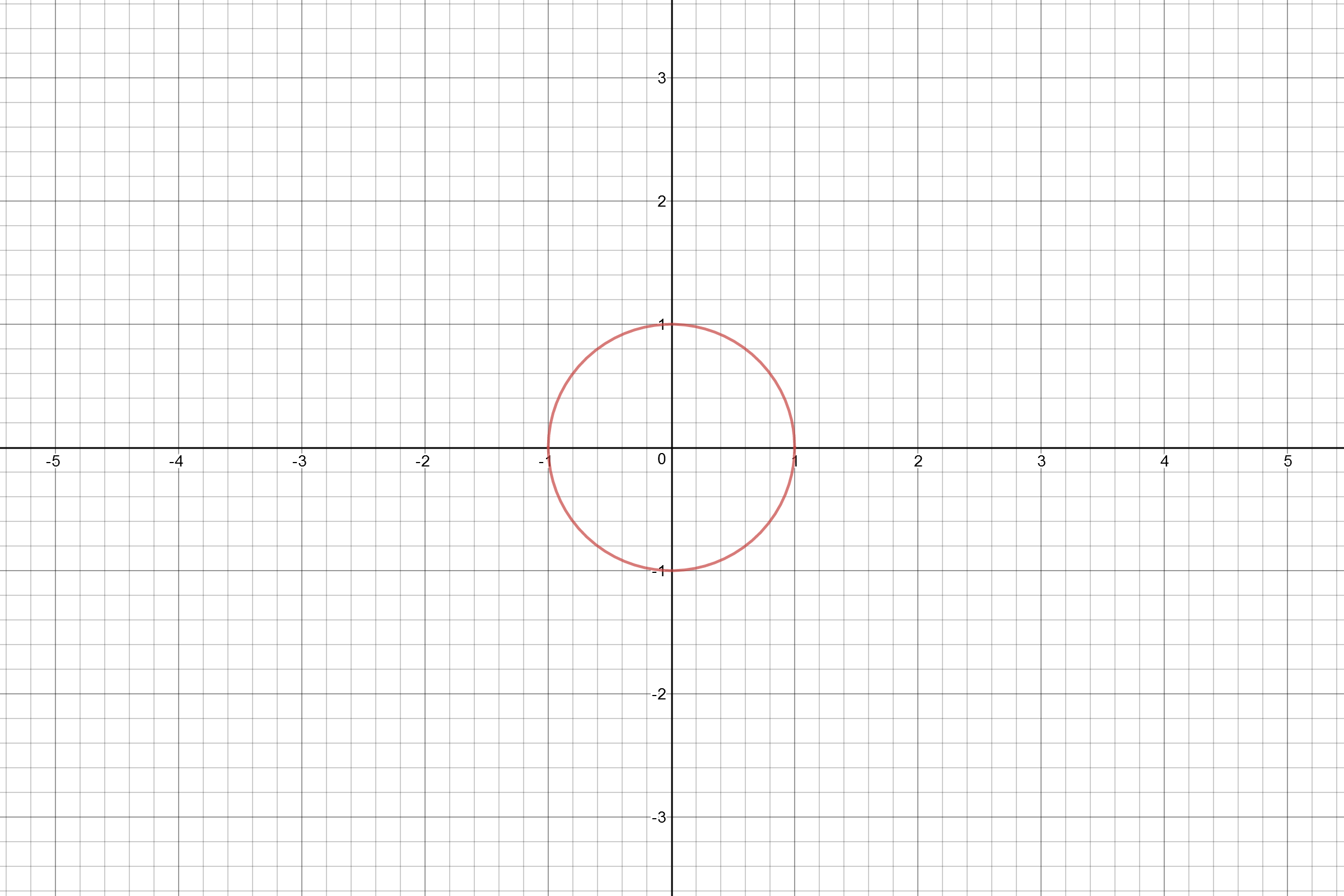}
\includegraphics[width=150pt]{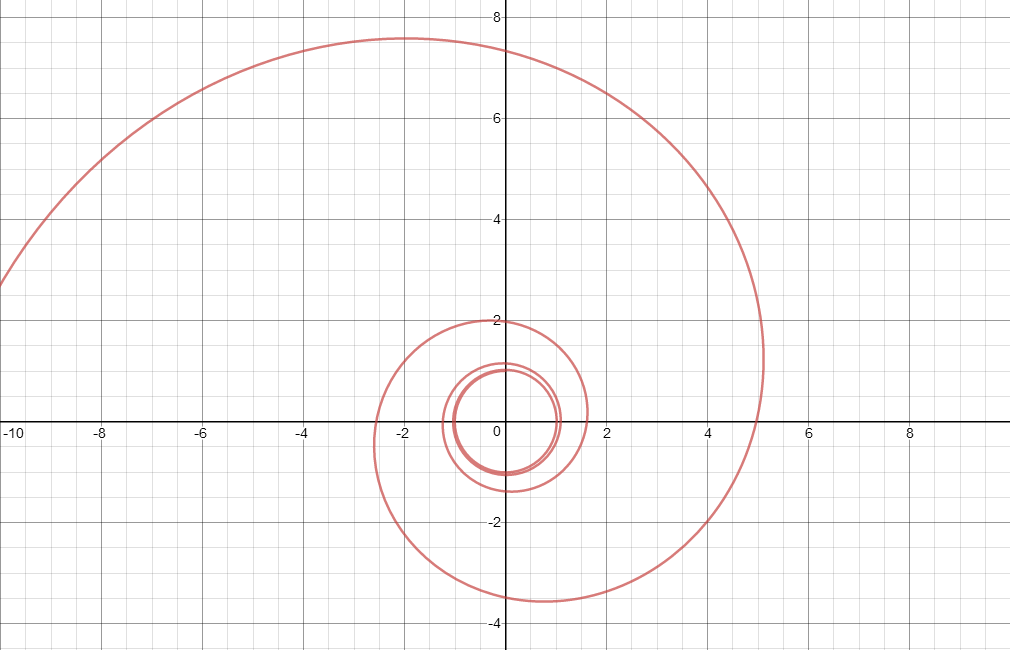}
\includegraphics[width=150pt]{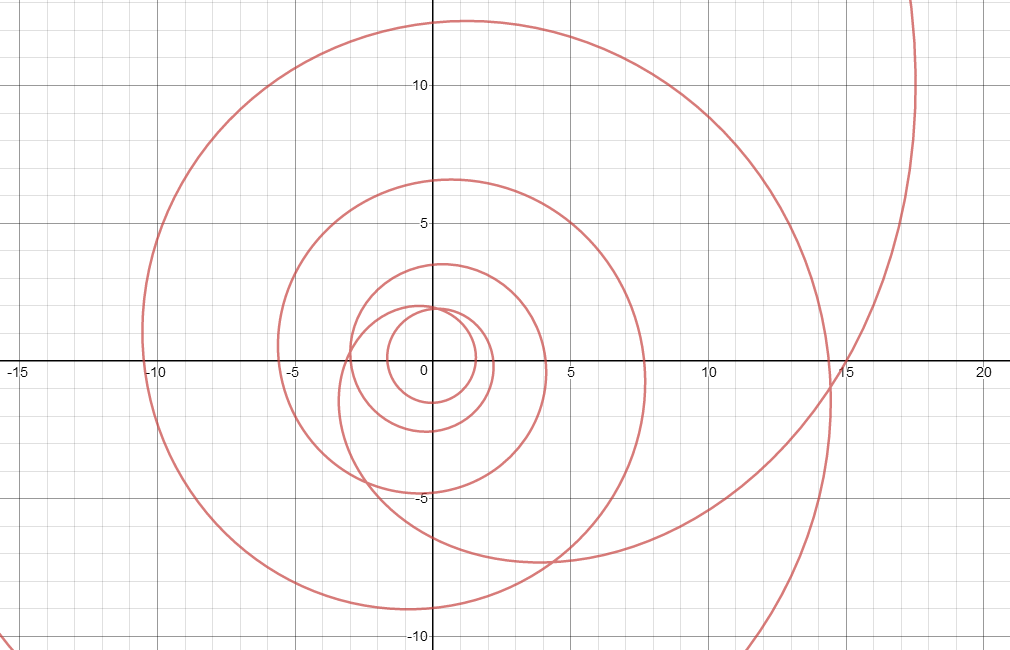}
\caption{Complete solutions}
\end{figure}
\bigskip
\bigskip
\bigskip
\bigskip

\section{Undercritical case: $c^2<4(1-d)$}
Let $K=\sqrt{1-d-\frac{c^2}{4}}$, we have
\begin{equation}
\theta=\frac{1}{K}\tan^{-1}(\frac{1}{K}(\tan\phi+\frac{c}{2})).
\end{equation}
Therefore, $\tan\phi+\frac{c}{2}=K\tan(K\theta)$. We have
\begin{equation}
\begin{split}
x&=\frac{1}{\sqrt{K^2\tan^2(K\theta)+K^2}}e^{(\frac{c}{2}+i)\theta}\big(i-(\frac{c}{2}-K\tan(K\theta))\big)\\
&=\frac{e^{(\frac{c}{2}+i)\theta}}{K}\big(\cos(K\theta)(i-\frac{c}{2})+K\sin(K\theta)\big)\\
&=\frac{e^{(\frac{c}{2}+i)\theta}}{2K}\big((e^{Ki\theta}+e^{-Ki\theta})(i-\frac{c}{2})-Ki(e^{Ki\theta}-e^{-Ki\theta})\big)\\
&=\frac{1}{2K}\big(e^{(\frac{c}{2}+(K+1)i)\theta}((1-K)i-\frac{c}{2})+e^{(\frac{c}{2}+(1-K)i)\theta}((1+K)i-\frac{c}{2})\big).
\end{split}
\end{equation}
We can scale the curve and express it as
\begin{equation}
\begin{split}
x=e^{(\frac{c}{2}+(K+1)i)\theta}((1-K)i-\frac{c}{2})+e^{(\frac{c}{2}+(1-K)i)\theta}((1+K)i-\frac{c}{2}).
\end{split}
\end{equation}

Note that $K>0$. The $e^{(\frac{c}{2}+(K+1)i)\theta}((1-K)i-\frac{c}{2})$ part will always have smaller magnitude and higher rotation speed. The $e^{(\frac{c}{2}+(1-K)i)\theta}((1+K)i-\frac{c}{2})$ part will always have larger magnitude and smaller angular speed. The shape will be closed to the spiral $e^{(\frac{c}{2}+(1-K)i)\theta}$. We can classify the shape of the curve as follows.

\begin{enumerate}
\item $K=1$, no angular movement:
In this case, 
\begin{equation}
\begin{split}
x&=e^{(\frac{c}{2}+2i)\theta}(-\frac{c}{2})+e^{\frac{c}{2}\theta}(2i-\frac{c}{2})
\end{split}
\end{equation}

\item $K<1$, bending outward.

\item $K>1$, bending inward.
\end{enumerate}

\begin{rmk}
The translating solution can be regarded as the limit case $(c,d)=(0,0)$, or equivalently $(c,k)=(0,1)$. The solution is the standard cycloid.
\end{rmk}

\begin{rmk}
All of the solution with $K\geq 1$ are shrinking solutions, that is, $d<0$.
\end{rmk}

\begin{figure}[ht]
\centering
\includegraphics[width=150pt]{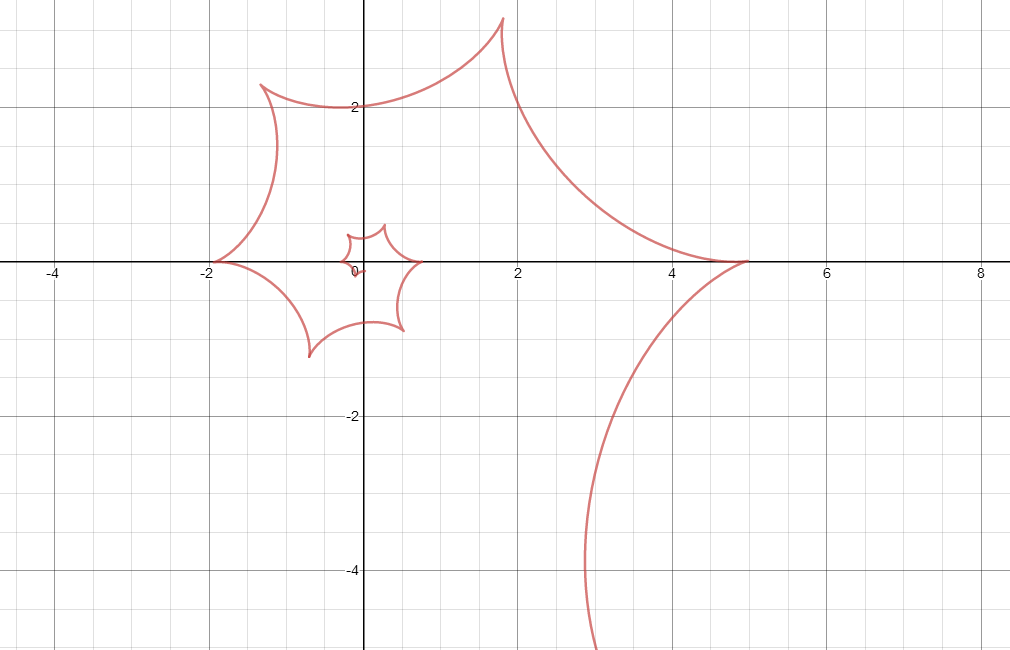}
\includegraphics[width=150pt]{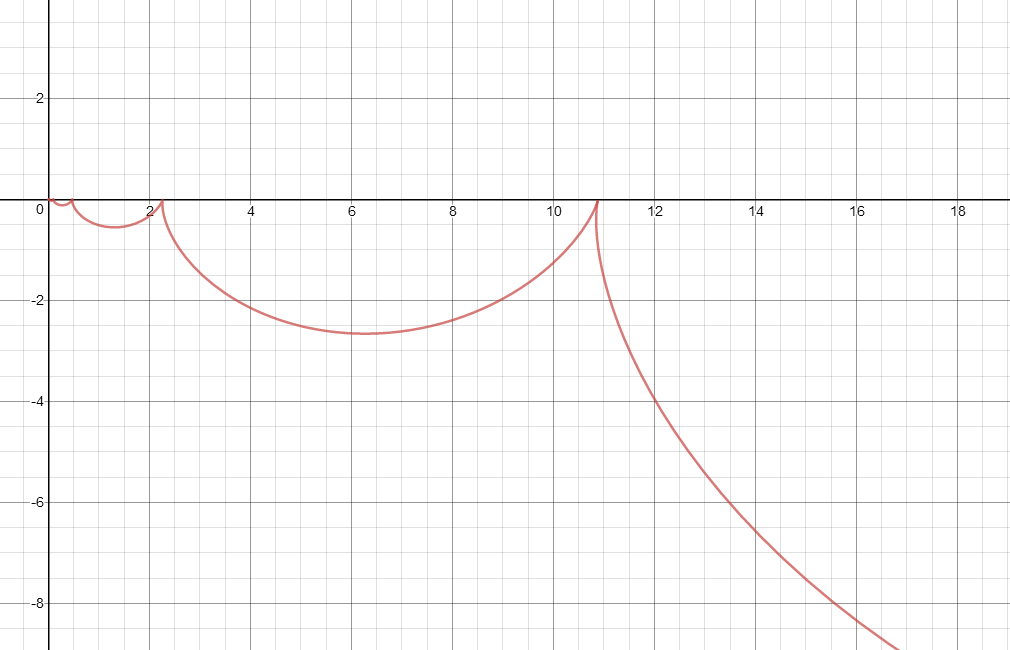}
\includegraphics[width=150pt]{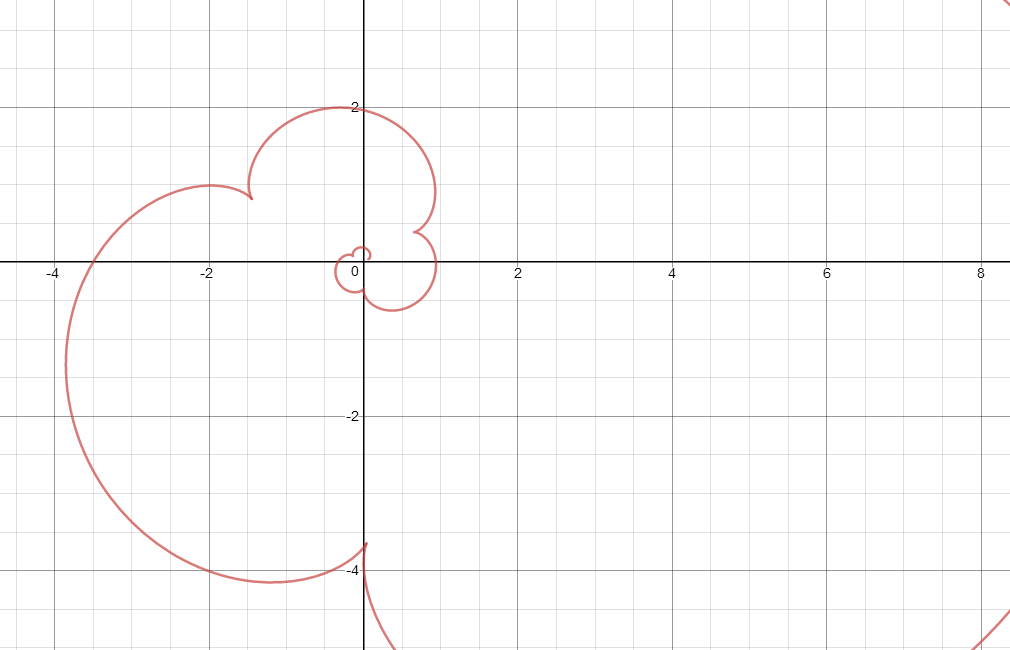}
\caption{The solutions where the cycloid bend inward, does not bend, bend outward}
\end{figure}

\subsection{Example: Numerical solution of the rotation case: $(c,d)=(1,0)$}

\begin{wrapfigure}{r}{0.23\textwidth}
\centering
\includegraphics[width=0.23\textwidth]{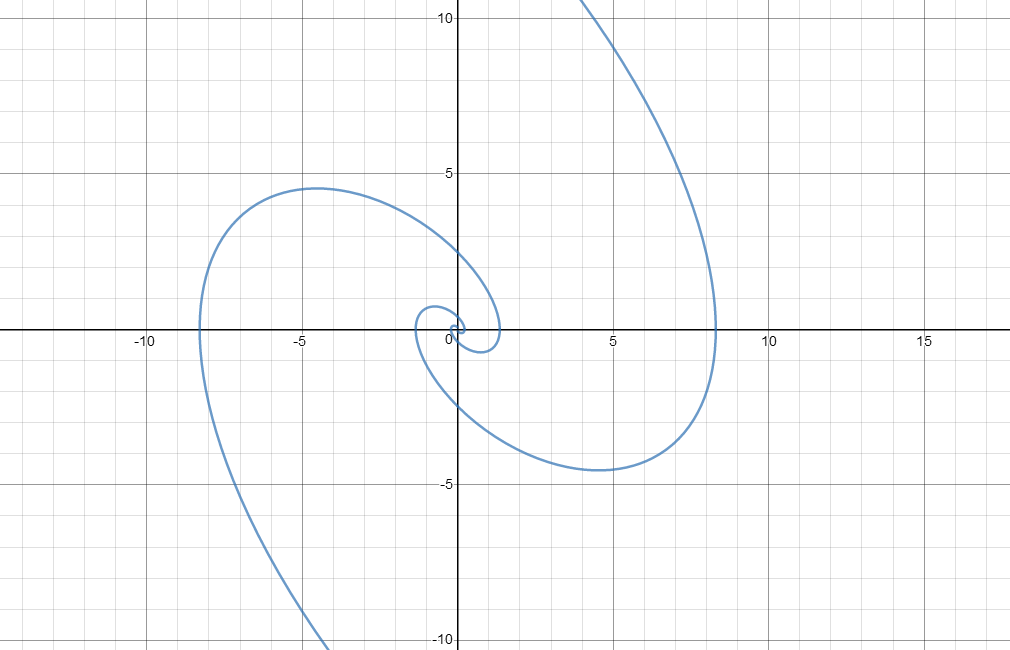}
\end{wrapfigure}

In this case, the relation of $\nu-\tau$ is given on the right. In this case, there are no straight lines which correspond to logarithmic spirals. Up to rotation and scaling, there are only 1 solution. On the curve, there are infinitely many cusps. The curve behaves like a cycloid scaled and rotated with respect to a logarithmic spiral.

\begin{figure}[ht]
\centering
\includegraphics[width=150pt]{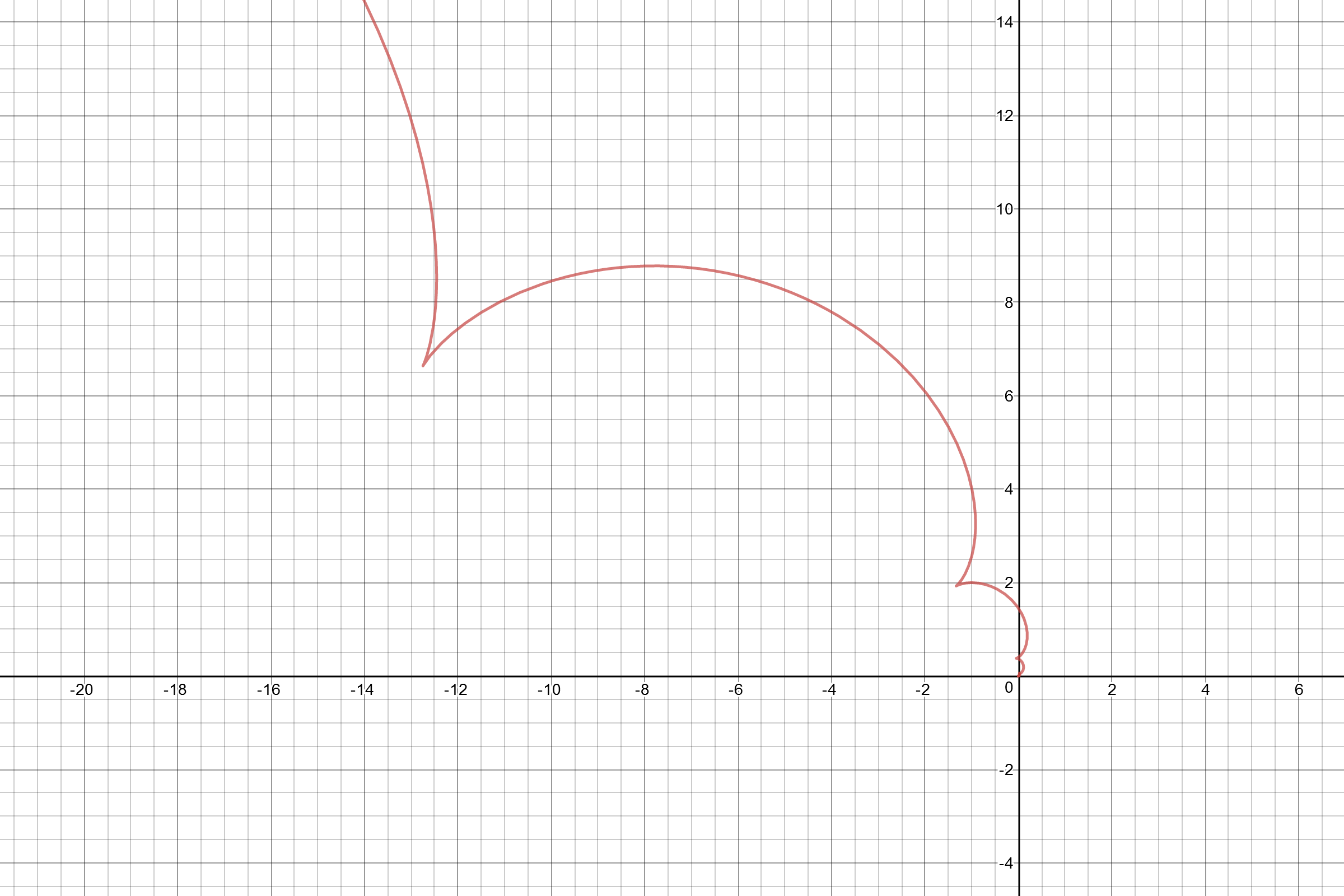}
\caption{the general solution in undercritical case}
\end{figure}

\section{Summary}

\begin{wrapfigure}{r}{0.23\textwidth}
\centering
\includegraphics[width=0.23\textwidth]{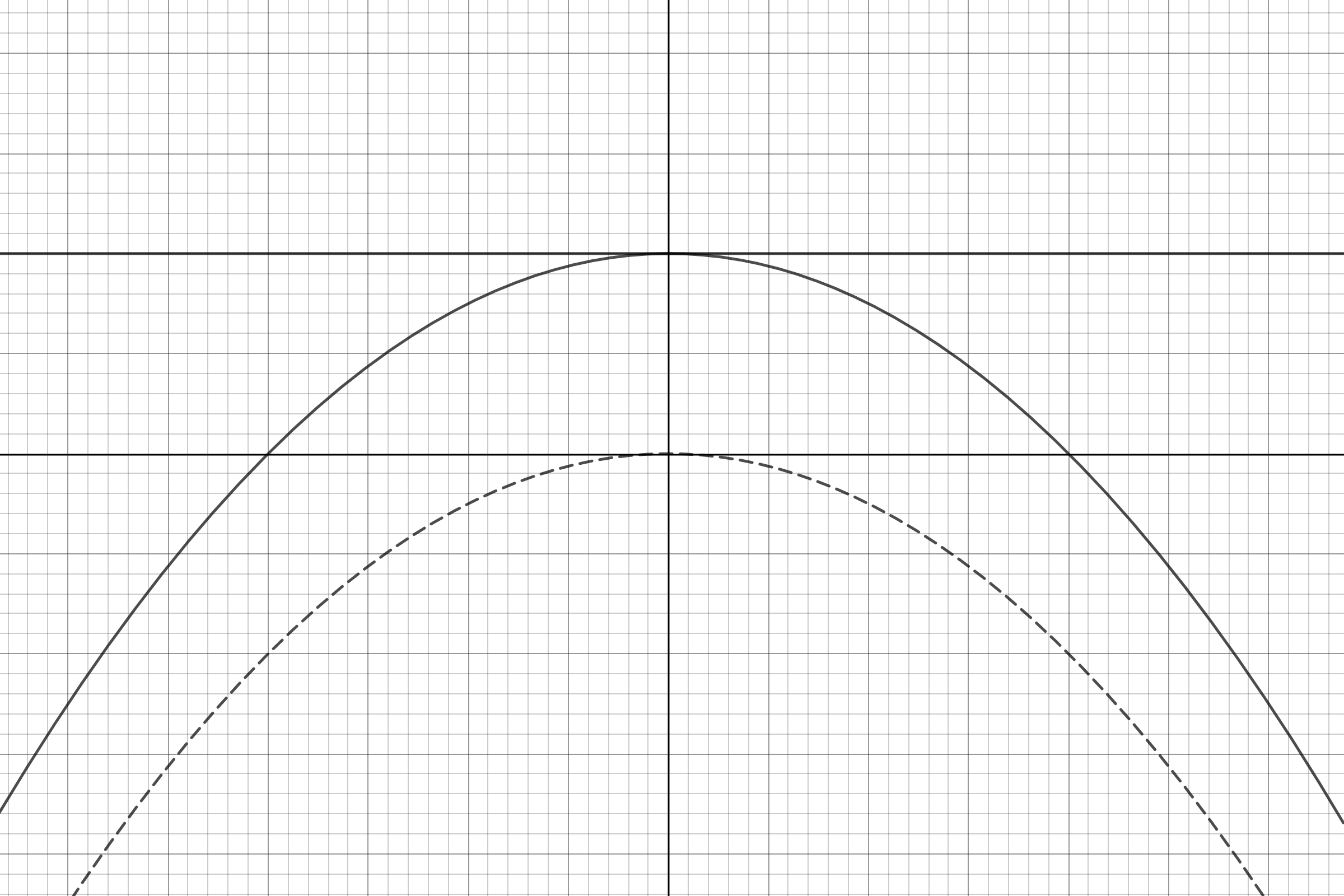}
\end{wrapfigure}
We can separate the $c-d$ space into different zones according to the behavior of the solutions. When $c^2<4(1-d)$, we have the undercritical case. When $c^2=4(1-d)$, we have the critical case. When $c^2>4(1-d)$, we have the overcritical case. Forthmore, for the case $d\geq1$, there is a unique complete solution.

\subsection{Undercritical case}
When $c^2<4(1-d)$, let $K=\sqrt{1-d-\frac{c^2}{4}}$, the solution is of the form
\begin{equation}
x=e^{(\frac{c}{2}+(K+1)i)\theta}((1-K)i-\frac{c}{2})+e^{(\frac{c}{2}+(1-K)i)\theta}((1+K)i-\frac{c}{2}).
\end{equation}

The solution behaves as a cycloid on a logarithmic spiral. We can further separate into 3 cases:
\begin{itemize}
\item $c^2<-4d$ The cycloid is inside the spiral.
\item $c^2=-4d$ The cycloid lies on a ray.
\item $-4d<c^2<4(1-d)$ The cycloid is outside the spiral.
\end{itemize}
\subsection{Critical case}
When $c^2=4(1-d)$, the solution is either the logarithmic spiral
\begin{equation}
x=e^{(\frac{c}{2}+i)\theta}
\end{equation}
or the curve
\begin{equation}
\begin{split}
x=(-\frac{c}{2}+i)\theta e^{(\frac{c}{2}+i)\theta}-e^{(\frac{c}{2}+i)\theta},
\end{split}
\end{equation}
which is a composition of $r=e^{\frac{c}{2}\theta}$ and $r=\theta e^{\frac{c}{2}\theta}$. 
\subsection{Overcritical case} When $c^2>4(1-d)$, let $K=\sqrt{\frac{c^2}{4}+d-1}$, the solution is either the logarithmic spirals
\begin{equation}
\begin{split}
x&=e^{(\frac{c}{2}+K+i)\theta},\\
x&=e^{(\frac{c}{2}-K+i)\theta},
\end{split}
\end{equation}
or the curves
\begin{equation}
\begin{split}
x=&e^{(\frac{c}{2}+K+i)\theta}(i-\frac{c}{2}-K)+e^{(\frac{c}{2}-K+i)\theta}(i-\frac{c}{2}+K),\\
x=&e^{(\frac{c}{2}+K+i)\theta}(i-\frac{c}{2}-K)-e^{(\frac{c}{2}-K+i)\theta}(i-\frac{c}{2}+K).\\
\end{split}
\end{equation}
According to the existence of complete solution, we can separate into the following 2 cases:
\begin{itemize}
\item $c^2>4(1-d)$, $d\leq1$. No complete solutions. 
\item $d\geq1$. Up to rotation and scaling, there exists a complete solution. 
Note that when $c=0$, the complete solution is derived in \cite{DLW} as
\begin{equation}
\begin{split}
x&=(-\alpha\sinh(\alpha\theta)\cos\theta-\cosh(\alpha\theta)\sin\theta,-\alpha\sinh(\alpha\theta)\sin\theta+\cosh(\alpha\theta)\cos\theta)
\end{split}
\end{equation}
even though they didn't point out it's complete without any cusp.
\end{itemize}

\end{document}